\documentclass{amsart}
\usepackage{pgf,amsthm,amscd,amssymb,verbatim,epsf,amsmath,eurosym,amsfonts,mathrsfs,graphicx,tikz-cd,lscape} 
\usepackage{amsthm,amscd,amssymb,verbatim,epsf,amsmath,amsfonts,mathrsfs,graphicx,dirtytalk,pdfpages,mdframed,xcolor,tikz-cd,xy}
\usepackage[colorlinks=true,linkcolor=blue,citecolor=blue]{hyperref}
\usepackage{epigraph}
\begin{document}
\theoremstyle{plain}
\newtheorem{Thm}{Theorem}
\newtheorem{Cor}[Thm]{Corollary}
\newtheorem{Ex}[Thm]{Example}
\newtheorem{Con}[Thm]{Conjecture}
\newtheorem{Lem}[Thm]{Lemma}
\newtheorem{Prop}[Thm]{Proposition}

\theoremstyle{definition}
\newtheorem{Def}[Thm]{Definition}
\newtheorem{Note}[Thm]{Note}
\newtheorem{Question}[Thm]{Question}

\theoremstyle{remark}
\newtheorem{notation}[Thm]{Notation}
\renewcommand{\thenotation}{}

\errorcontextlines=0
%\numberwithin{equation}{section}
%\numberwithin{Thm}{section}
\renewcommand{\rm}{\normalshape}%
\newcommand{\transv}{\mathrel{\text{\tpitchfork}}}
\makeatletter
\newcommand{\tpitchfork}{%
  \vbox{
    \baselineskip\z@skip
    \lineskip-.52ex
    \lineskiplimit\maxdimen
    \m@th
    \ialign{##\crcr\hidewidth\smash{$-$}\hidewidth\crcr$\pitchfork$\crcr}
  }%
}
\makeatother

\title[The three obdurate conjectures of differential geometry]{The three obdurate conjectures \\ of differential geometry}

\author{Brendan Guilfoyle}
\address{Brendan Guilfoyle\\
          School of Science Technology Engineering and Mathematics\\
          Munster Technological University, Kerry\\
          Tralee\\
          Co. Kerry\\
          Ireland.}
\email{brendan.guilfoyle@mtu.ie}
\author{Wilhelm Klingenberg}
\address{Wilhelm Klingenberg\\
 Department of Mathematical Sciences\\
 University of Durham\\
 Durham DH1 3LE\\
 United Kingdom}
\email{wilhelm.klingenberg@durham.ac.uk }

\begin{abstract}
%{
We explore the role of symmetry in three obdurate conjectures of differential geometry: the Carath\'eodory, the Willmore and the Lawson Conjectures. All three Conjectures concern surfaces in 3-dimensional space-forms, which have a high degree of symmetry. It is shown that this symmetry is broken and more general ambient metrics are considered, none of the Conjectures continue to hold. 

The subtle manner in which symmetry enters the first Conjecture is also explained in detail. 
%}
\end{abstract} 

\date{\today}
\maketitle
%\tableofcontents
%\newpage
In a 1980s UC Berkeley lecture \cite{KMS16}, Robert Osserman identified what he referred to as the three obdurate conjectures of differential geometry, all of which were open at that time: the Lawson Conjecture \cite{Lawson70}, which was proven in 2014 by {Marques} and Neves \cite{MaN14}, the Willmore Conjecture \cite{Will65} which was proven in 2013 by Brendle \cite{Br13}, and the Carath\'eodory Conjecture, whose proof has been published by the authors in \cite{GK11} (see \cite{GK19} \cite{GK20} \cite{GK24}).

The three Conjectures were:

\vspace{0.1in}
\begin{itemize}
    \item[]{\bf Carath\'eodory Conjecture (1924):} \textit{Every strictly convex surface in ${\mathbb R}^3$ contains at least two umbilic points ({points} at which the principal curvatures are equal)}.
    \item[]
    \item[]{\bf Willmore Conjecture (1969{):}} \textit{The Willmore energy of an embedded torus $T$ in ${\mathbb S}^3$ satisfies 
    \vspace*{-0.05in} 
    $${\mathcal W}(T)\equiv\iint_T 1+H^2 d\mu\geq 2\pi^2,$$
    \vspace*{-0.1in} 
   where $H$ is the mean curvature of $T$.}
    \item[]
    \item[]{\bf Lawson Conjecture (1970):} \textit{Up to ambient isometry, the only embedded minimal torus in ${\mathbb S}^3$ is the Clifford torus}.
\end{itemize} 
\vspace{0.1in}

All three Conjectures are concerned with surfaces embedded in  3-manifolds of constant curvature (namely ${\mathbb R}^3$ or ${\mathbb S}^3$) and each has a sizeable history, which we will refrain from entering. 

The purpose of this note is, rather, to demonstrate that the underlying reason the Conjectures are true is the size of the ambient isometry groups.  This is done by constructing metrics close to the constant curvature metric, along with embedded surfaces for which the Conjectures do not hold. Perhaps it should come as no surprise that the remaining unresolved conjectures in such an old field involve properties that are unique to the geometric setting. 

The role of symmetry is most subtle in the case of the Carath\'eodory Conjecture and this may well explain its longevity. In Section \ref{s:1.1a} we describe the symmetry argument that underlies the reason that the Carath\'eodory Conjecture is true. Section \ref{s:1.2} turns to the parabolic evolution equations that prove the existence of solutions to the $\bar{\partial}$-equation in certain pseudo-K\"ahler manifolds, thus leading to the proof of the Conjecture when compared with the symmetry properties.

\vspace{0.1in}
%%%%%%%%%%%%%%%%%%%%%%%%%%%%%%%%%%%%%%%%%%%%%%%%%%%%%%%%%%%%%%%%%%%%%%%%%%%%%%%%%%%%%%%%%%%%%%%%%%%%%%%%%%%%%%%%%%%%
\section{The Symmetry of the Ambient Spaces}\label{s:1.1}
%%%%%%%%%%%%%%%%%%%%%%%%%%%%%%%%%%%%%%%%%%%%%%%%%%%%%%%%%%%%%%%%%%%%%%%%%%%%%%%%%%%%%%%%%%%%%%%%%%%%%%%%%%%%%%%%%%%%

\begin{Thm}\label{t:0}
    There exist Riemannian metrics $g$ on ${\mathbb R}^3$ and ${\mathbb S}^3$ that are arbitrarily close to the constant curvature metric $g_0$ in the $C^0$ or $L_2$ norm:
    \[
    \|g-g_{0}\|^2\leq\epsilon \qquad{\mbox{ for } } \epsilon>0,
    \]
    and smoothly embedded surfaces for which none of the three Conjectures hold.
\end{Thm}
\begin{proof}
    That the Carath\'eodory Conjecture may not hold for nearly Euclidean metrics was established in Theorem 2 of \cite{G20}:  \textit{for all $\epsilon>0$, there exists a smooth Riemannian metric $g$ on ${\mathbb R}^3$ and a smooth strictly convex 2-sphere $S\subset{\mathbb R}^3$ such that $S$ has a single umbilic point and $\|g-g_{0}\|^2\leq\epsilon$.} Here $g_0$ is the flat metric on ${\mathbb R}^3$ and $\|,\|$ is the $L_2$-norm.

    The metrics are obtained by fixing a round sphere  $S$ in flat ${\mathbb R}^3$ and then deforming the Euclidean metric, leaving the sphere fixed in such a manner that the metric induced on $S$ is preserved. The curvature introduced deforms only the second fundamental form and the perturbations are sufficient to control the principal foliation of $S$ and create a single umbilic point. Bump functions with support in a collared neighbourhood of $S$ bring the metric $L_2$ close to $g_0$.

    The following argument shows that the same is true for the Willmore Conjecture: \textit{for all $\epsilon>0$, there exists a smooth Riemannian metric $g$ on ${\mathbb S}^3$ and a smoothly embedded torus $T$ such that the Willmore energy of $T$ with respect to $g$ satisfies 
    \[
    {\mathcal W}(T,g)\equiv\iint_T 1+H^2 d\mu_g<2\pi^2,
    \]
    and $\|g-g_{1}\|^2\leq\epsilon$.} Here $g_1$ is the round metric on ${\mathbb S}^3$ and $H$ is the mean curvature of $T$. To this end, consider the following perturbation $g_\epsilon$ of the round metric on ${\mathbb S}^3$ in Hopf coordinates $(\rho,\theta_1,\theta_2)$:
    \[
    g_\epsilon=d\rho^2+\sin^2\rho\; d\theta_1^2+\cos^2\rho \;d\theta_2^2+2\epsilon\sin\rho \cos\rho\; d\theta_1d\theta_2, \qquad\qquad   {0\leq\epsilon<1}.
    \]
    For $\epsilon=0$ this is the round metric, and the embedded torus given by $\rho=\frac{\pi}{4}$ is a Clifford torus which we denote $T$. Now deform the metric $g$ by allowing $\epsilon>0$ and keep the torus $T$ fixed. {For $\epsilon<1$ the metric remains Riemannian.}
    
    The unit normal vector to $T$ is clearly the gradient of $N=d\rho$ and so the induced metric is
    \[
    g|_T={\textstyle{\frac{1}{2}}}\left[ d\theta_1^2+2\epsilon d\theta_1 d\theta_2+d\theta_2^2\right],
    \]
    with area form $d\mu={\textstyle{\frac{1}{2}}}\left( 1-\epsilon^2\right)^{\scriptstyle{\frac{1}{2}}} d\theta_1 d\theta_2$.

    Moreover, the second fundamental form is
    \[
    A_{ij}=\nabla_iN_j=-\Gamma^{\rho}_{ij}={\textstyle{\frac{1}{2}}}\left. \frac{\partial g_{ij}}{\partial\rho}\right|_{\rho=\pi/4}=\left[\begin{matrix}
        1 & 0 \\
        0 & -1
    \end{matrix}\right].
    \]
    Thus, under this deformation $T$ remains minimal $(H=g^{ij}A_{ij}=0)$ and for $\epsilon>0$
    \[
    {\mathcal W}(T,g)\equiv\iint_T 1+H^2 d\mu_g=\frac{1}{2}(1-\epsilon^2)^{\scriptstyle{\frac{1}{2}}}\int_0^{2\pi}d\theta_1\int_0^{2\pi}d\theta_2=2(1-\epsilon^2)^{\scriptstyle{\frac{1}{2}}}\pi^2<2\pi^2,
    \]
     as claimed. While $g$ is only $C^0$ close to the round metric, this can be extended to $L_2$ by using bump functions in a neighbourhood of $T$ {as follows.

Introduce a function $\Psi:{\mathbb R}\rightarrow{\mathbb R}$ into the metric so that
 \[
    g_\epsilon=d\rho^2+\sin^2\rho\; d\theta_1^2+\cos^2\rho \;d\theta_2^2+2\epsilon \Psi(\rho)\sin\rho \cos\rho\; d\theta_1d\theta_2, \quad\quad  0\leq\epsilon<1.
\]
where $\Psi$ is defined:
\[
\Psi(\rho)=\left\{ \begin{matrix}
1&\qquad\qquad {\mbox{for }}|\rho-\pi/4|\leq \epsilon/4\\
f(\rho)&\qquad\qquad {\mbox{for }}\epsilon/4\leq|\rho-\pi/4|\leq \epsilon/2\\
0&\qquad\qquad {\mbox{for }} \epsilon/2\leq|\rho-\pi/4|
\end{matrix}\right.,
\]
\vspace{0.1in}
where $f(\rho)$ is a smooth bump function with $0\leq f(\rho)\leq1$. 

Clearly the torus $T$ given by $\rho=\pi/4$ has the same Willmore energy as before, namely ${\mathcal W}(T,g_\epsilon)=2(1-\epsilon^2)^{\scriptstyle{\frac{1}{2}}}\pi^2$ and the $L_2$ difference between the flat metric and this bumped metric is
\[
\|g-g_{\epsilon}\|^2=2\iiint_{{\mathbb S}^3}|g-g_{\epsilon}|^2d^3V
=4\epsilon^2\int_0^{2\pi}d\theta_1\int_0^{2\pi}d\theta_2\int_{\pi/4-\epsilon/2}^{\pi/4+\epsilon/2}\Psi(\rho)^2d\rho \leq16\pi^2\epsilon^3,\nonumber
\]
which can be made arbitrarily small.}

     Finally, the Lawson Conjecture trivially fails in the absence of symmetry, since from the work of Marques, Neves and Song \cite{MaN17} \cite{Song23}, every closed 3-manifold admits an infinite number of embedded minimal tori, and without ambient isometries, obviously these minimal surfaces are non-isometric.
\end{proof}
\vspace{0.1in}

The above arguments can be extended to show that Hamburger's real analytic index bound {$i(p)\leq1$ for any isolated umbilic point $p$,} can also be violated by symmetry breaking  \cite{G20} \cite{Ham41}.

\vspace{0.1in}
%%%%%%%%%%%%%%%%%%%%%%%%%%%%%%%%%%%%%%%%%%%%%%%%%%%%%%%%%%%%%%%%%%%%%%%%%%%%%%%%%%%%%%%%%%%%%%%%%%%%%%%%%%%%%%%%%%%%
\section{The Carath\'eodory Conjecture}\label{s:1.1a}
%%%%%%%%%%%%%%%%%%%%%%%%%%%%%%%%%%%%%%%%%%%%%%%%%%%%%%%%%%%%%%%%%%%%%%%%%%%%%%%%%%%%%%%%%%%%%%%%%%%%%%%%%%%%%%%%%%%%
The central geometric idea of the proof is that the Carath\'eodory Conjecture can be stated and proven using the following \cite{GK19} \cite{GK20} \cite{GK24}: an oriented surface $S$ in ${\mathbb R}^3$ gives rise through its oriented normal lines to a surface $\Sigma$ in the 4-dimensional space of all oriented lines, which can be identified with the total space $T{\mathbb S}^2$ of the tangent bundle of the 2-sphere. If $S$ is closed and strictly convex, then $\Sigma$ is the graph of a global section of the bundle $\pi: T{\mathbb S}^2\rightarrow {\mathbb S}^2$, where $\pi$ takes an oriented line to its direction. 

Moreover, there are natural Euclidean-invariant geometric structures on $T{\mathbb S}^2$, including a symplectic form (a closed non-degenerate 2-form) $\omega$ and a complex structure ${\mathbb J}$ (an endomorphism of the tangent space that squares to minus the identity). The surface $\Sigma\subset T{\mathbb S}^2$ is Lagrangian ($\omega|_\Sigma=0$) and an umbilic point on $S$ corresponds to a complex point on $\Sigma$ - a point where the ambient complex structure preserves the tangent space of $\Sigma$. Thus we see that being umbilic is a property of the normal lines: a point on a surface is umbilic iff points on all parallel surfaces are umbilic. 

One is led to the following reformulation of the Carath\'eodory Conjecture: \textit{every global Lagrangian section of $T{\mathbb S}^2$ contains at least two complex points}. The aspect that puts this outside the standard framework of K\"ahler geometry is that the complex structure and the symplectic structure, while compatible:
\[
\omega(\cdot,\cdot)=\omega({\mathbb J}\cdot,{\mathbb J}\cdot),
\] 
define \textit{opposite} orientations on the ambient space. 

Thus the associated metric ${\mathbb G}(\cdot,\cdot)=\omega({\mathbb J}\cdot,\cdot)$ on $T{\mathbb S}^2$ is not Riemannian: it is pseudo-K\"ahler of neutral signature $(2,2)$. In the language of Gromov, the complex structure is not tamed by the symplectic structure - if it did, it would be impossible to have any complex points on a Lagrangian surface!  Indeed, invariant metrics on spaces of submanifolds are generally pseudo-Riemannian and so this is unavoidable \cite{AGK11} \cite{GK05} \cite{Salvai05} \cite{Salvai07}. 
 
In this pseudo-K\"ahler case, the remarkable fact emerges that the dimension of the space of graphical holomorphic discs with boundary lying on a fixed Lagrangian local section $\Sigma\subset T{\mathbb S}^2$ is related to the number of complex points contained within the intersection on $\Sigma$, and therefore, to the number of umbilic points contained on the underlying surface $S\subset{\mathbb R}^3$. More precisely, the linearization of the $\bar{\partial}$-operator has finite kernel and cokernel, and index given by
\[
I={\mbox{dim ker }}\nabla \bar{\partial}-{\mbox{dim coker }}\nabla \bar{\partial}=\mu(T{\mathbb S}^2,T\Sigma)+2,
\]
where $\mu$ is the Keller-Maslov index of the boundary curve. In the case of a section, this counts the number of umbilic points through the formula {(Proposition 11 of} \cite{GK04})
\[
\mu(T{\mathbb S}^2,T\Sigma)=4i,
\]
where $i\in{\mathbb Z}/2$ is the sum of the umbilic indices inside the boundary curve. 

In the K\"ahler case, a compact Lagrangian surface $\Sigma$ is totally real and so is a good boundary condition for the $\bar{\partial}$-operator, and it can be shown to be Fredholm regular {(Theorem I of} \cite{Oh96}). That is, under a small perturbation of the boundary condition $\Sigma$, the cokernel of the linearization vanishes and the Keller-Maslov class determines the dimension of the space of holomorphic discs. In the pseudo-K\"ahler case, being Lagrangian does not imply being totally real - indeed, for topological reasons an embedded sphere \textit{must} have at least one complex point. The existence of complex points on $\Sigma$ poses both a difficulty and an added richness not seen in the K\"ahler setting. 

The Euclidean group acts isometrically on the pseudo-K\"ahler structure and if a Lagrangian section $\Sigma$ has a single complex point, then quotienting out by the stabilizer subgroup of the point, it is again Fredholm regular {(Theorem 1.1 of} \cite{GK20}). There should therefore be no holomorphic discs with boundary lying on a totally real Lagrangian section - in particular, on a totally real Lagrangian hemisphere. In this case the umbilic index sum is $i=0$, and the space of unparameterized holomorphic discs (subtracting the dimension of the M\"obius group of the disc) is $I-3=\mu-1=4i-1=-1$. 

Proving the existence of non-trivial holomorphic discs with boundary lying on any totally real Lagrangian hemisphere would therefore disprove the existence of a counter-example to the Carath\'eodory Conjecture. Precisely the same argument establishes a Conjecture of Toponogov \cite{GK24} \cite{Top95} which states that\textit{ the curvatures $\kappa_1,\kappa_2$ of a complete embedded plane $P$ in ${\mathbb R}^3$ satisfy}
\[
\inf_{p\in P}|\kappa_1(p)-\kappa_2(p)|=0.
\]
Toponogov proves that the Gauss image of such a plane is a hemisphere and thus umbilic-free hemispheres arise naturally in this context, and existence of a holomorphic disc attached to any totally real Lagrangian hemisphere also proves Toponogov's Conjecture. The global symmetry argument is slightly easier in this case, as a putative counterexample would be totally real, even at infinity.

Furthermore, a topological extension of these arguments leads to the local index bound {$i(p)<2$ for any isolated umbilic point $p$} \cite{GK24a}. There, the global section $\Sigma$ of $T{\mathbb S}^2$ is replaced by a 2-sphere with crosscaps to cancel hyperbolic complex points. It is a remarkable fact that this yields a bound on the umbilic index $i<2$, a bound that is weaker than the real analytic index bound of Hamburger \cite{Ham41}.

\vspace{0.1in}
%%%%%%%%%%%%%%%%%%%%%%%%%%%%%%%%%%%%%%%%%%%%%%%%%%%%%%%%%%%%%%%%%%%%%%%%%%%%%%%%%%%%%%%%%%%%%%%%%%%%%%%%%%%%%%%%%%%%
\section{Holomorphic Discs}\label{s:1.2}
%%%%%%%%%%%%%%%%%%%%%%%%%%%%%%%%%%%%%%%%%%%%%%%%%%%%%%%%%%%%%%%%%%%%%%%%%%%%%%%%%%%%%%%%%%%%%%%%%%%%%%%%%%%%%%%%%%%%

In a 4-dimensional K\"ahler manifold $(M,\omega,{\mathbb J},{\mathbb G})$, holomorphic surfaces are minimal: their mean curvature vector vanishes. This is also true in the pseudo-K\"ahler setting, {where surfaces with vanishing mean curvature are called maximal. In this case the metric induced on a surface can be degenerate and defining the mean curvature vector problematic. }

The fundamental difference between the definite ($\varepsilon=1$) and neutral ($\varepsilon=-1$) metrics is seen in the Wirtinger identities:
\[
\omega(X_1,X_2)^2+\varepsilon {\mbox{ det }}[X_1,X_2,{\mathbb J}X_1,{\mathbb J}X_2]={\mbox{ det }}{\mathbb G}(X_i,X_j).
\]
{In the neutral case the metric induced on the tangent plane to a surface $\Sigma\subset M$ can be positive definite, negative definite, Lorentz, degenerate or totally null. A Lagrangian plane must be Lorentz  or totally null, while a holomorphic plane must be definite or totally null. Totally null planes are both Lagrangian and holomorphic, and arise at complex points on Lagrangian surfaces and at Lagrangian points on holomorphic surfaces \cite{GKR10}.} For more details on submanifolds in pseudo-Riemannian spaces see \cite{Anciaux11}, for the specific case of surfaces in neutral 4-manifolds see \cite{GK05} \cite{GK08}.

To prove that there exist holomorphic discs with boundary lying on any Lagrangian totally real hemisphere, start with a definite disc and evolve it using mean curvature flow, the objective being to converge to a disc that is not only maximal, but holomorphic. The challenges faced by such an approach include: 
\begin{itemize}
    \item[1.] the codimension is two and therefore the flow is a system of equations, 
    \item[2.] since the symplectic form is globally exact $\omega=d\Phi$, the area of a smooth holomorphic disc $D$ with boundary on a Lagrangian sphere $\Sigma$ is zero:
    \[
    \iint_DdA=\iint_D\omega=\int_{\partial D} \Phi=\iint_\Sigma \omega=0,
    \]
    \item[3.] the evolving disc must remain definite throughout the flow,
    \item[4.] appropriate parabolic boundary conditions must be implemented,
    \item[5.] the boundary of the flowing surface must remain in the hemisphere,
    \item[6.] the flow must converge (albeit weakly) to a holomorphic, rather than maximal, disc.
\end{itemize}

These challenges are dealt with in the following ways. Regarding Challenge 1., prior to 2008 higher codimension mean curvature had not been extensively considered \cite{AmbS97} \cite{ChenLi01}  \cite{LiS09}, but subsequent work has shown it to be tractable - see for example \cite{AaS13} \cite{HaS24} \cite{Liu18} \cite{LaN24} \cite{Naff23} and the survey paper \cite{Smoc11}. 

In the Riemannian case most of the attention in mean curvature flow is on the singularities that inevitably form due to the shrinking nature of the motion. In the indefinite case, however, the area of a positive surface \textit{increases} under mean curvature flow, so the evolution tends to behave better if it can be constrained to a compact region. Indeed, when flowing positive graphs regularity theory is simpler and more natural in the pseudo-Riemannian setting than in the Riemannian one \cite{LiS11}. 

Given a smooth umbilic-free hemisphere, that is, a surface $S$ without umbilic points with Gauss area a hemisphere, and the associated totally real Lagrangian hemisphere $\Sigma\subset T{\mathbb S}^2$, 
consider then a family of positive sections $f_t:D\rightarrow T{\mathbb S}^2$ such that
\[
\frac{d f}{dt}^\bot=H,
\]
with initial and boundary conditions:
\begin{enumerate}
\item[(i)] $f_0(D)=D_0,$
\item[(ii)]$f_t(\partial D)\subset \tilde{\Sigma}$,
\item[(iii)] the hyperbolic angle $B$ between $Tf_t(D)$ and $T\tilde{\Sigma}$ is constant along $f_t(\partial D)$,
\item[(iv)] $f_t(\partial D)$ is asymptotically holomorphic: $|\bar{\partial}f_t|=C/(1+t)$,
\end{enumerate} 
\vspace{0.1in}
where $H$ is the mean curvature vector of $f_t(D)$, and $D_0$ and $\tilde{\Sigma}$ are given positive sections, and $C$ a non-zero constant.

The surface $\tilde{\Sigma}$ is obtained by adding a linear holomorphic twist to $\Sigma$ centred at its north pole. This contactification of the disc is precisely what is required to make the induced metric definite on the deformed surface $\tilde{\Sigma}$ {(Proposition 5.6 of} \cite{GK24}).  As the deformation is holomorphic, it does not introduce complex points and the maximum domain it can positivize is an open hemisphere. Thus, while we have addressed Challenge 2. by ensuring we can attach a definite disc to start, in fact we get an area formula for any smooth holomorphic disc with boundary lying on the twisted surface.

Given an initial positive disc $D_0$, Challenge 3. is dealt with by a priori estimates using the maximum principle both in the interior and at the boundary - this ensures that the flow remains parabolic {(Theorem 6.1 of} \cite{GK24}). The boundary conditions consist of Dirichlet conditions (ii) and Neumann conditions (iii) and (iv). Constant angle conditions have been considered for mean curvature flow \cite{Alt96} \cite{Freire10} \cite{Stahl96a} \cite{Stahl96b}, but only in codimension one and not between intersecting positive surfaces in a neutral 4-manifold. We are free to chose the size of this hyperbolic angle and by choosing it large enough we can increase the stickiness of the boundary and stop it from running off the edge of the hemisphere {(Theorem 6.3 of } \cite{GK24}), which is Challenge 5.

The condition (iv) says that the flowing surface becomes holomorphic at the boundary as $t\rightarrow\infty$ and this ensures that the flow converges to a holomorphic rather than just a maximal disc - Challenge 6.  These result ultimately in the flow weakly converging to a holomorphic disc with boundary lying on an arbitrary Lagrangian hemisphere - Theorem 6.17 of \cite{GK24}.

Comparing this with the Fredholm property in Theorem 1.1 of \cite{GK20} which relies on the ambient symmetry, we conclude that a totally real Lagrangian hemisphere cannot be closed with the creation of a single complex point, which proves the Carath\'eodory Conjecture.

\vspace{0.1in}
%\noindent{\bf Statements and Declarations:}
%
%The authors have no financial or non-financial interests that are directly or indirectly related to this work. No data was collected during the course of this work. 


\begin{thebibliography}{9}

\bibitem{AGK11}
D.V. Alekseevsky, B. Guilfoyle and W. Klingenberg, {\it On the geometry of spaces of oriented geodesics}, Ann. Global Anal. Geom. {\bf 40.4} (2011) 389--409. Erratum: Ann. Global Anal. Geom. {\bf 50.1} (2016) 97--99.

\bibitem{Alt96}
S. Altschuler and F.W. Lang, {\it Translating surfaces of the non-parametric mean curvature flow with prescribed contact angle}, Calc. Var. Partial Differential Equations {\bf 2.1} (1994) 101--111.

\bibitem{Anciaux11}
H. Anciaux, Minimal submanifolds in pseudo-Riemannian geometry, World Scientific, New Jersey, 2011.

\bibitem{AmbS97}
L. Ambrosio and H. M. Soner, {\it A measure-theoretic approach to higher codimension mean curvature flows}, Ann. Scuola Norm. Sup. Pisa Cl. Sci. (4) {\bf 25.1} (1997) 27--49.

\bibitem{AaS13}
C. Arezzo and J. Sun, {\it Self-shrinkers for the mean curvature flow in arbitrary codimension},  Math. Z. {\bf 274.3} (2013) 993--1027.

\bibitem{Br13}
S. Brendle, {\it Embedded minimal tori in $S^3$ and the Lawson conjecture}, Acta Math. {\bf 2.211} (2013) 177--190.

\bibitem{ChenLi01}
J. Chen and J. Li, {\it Mean curvature flow of surface in 4-manifolds}, Adv. Math. {\bf 163.2} (2001) 287–-309.

\bibitem{Freire10}
A. Freire, {\it Mean curvature motion of graphs with constant contact angle at a free boundary}, Anal. PDE {\bf 3.4} (2010) 359--407.

\bibitem{G20}
B. Guilfoyle, {\it On isolated umbilic points}, Comm. Anal. Geom. {\bf 28.8} (2020) 2005--2018.

\bibitem{GKR10}
B. Guilfoyle, M. Khalid and J. J. Ram\'on Mar\'i, {\it Lagrangian curves on spectral curves of monopoles}, Math. Phys. Anal. Geom. {\bf 13.3} (2010) 255--273.

\bibitem{GK04}
B. Guilfoyle and W. Klingenberg, {\it Generalised surfaces in ${\mathbb{R}}^3$}, Math. Proc. R. Ir. Acad. {\bf 104A} (2004) 199--209.

\bibitem{GK05}
B. Guilfoyle and W. Klingenberg, {\it An indefinite K\"ahler metric on the space of oriented lines}, J. London Math. Soc. {\bf 72.2} (2005) 497--509. 

\bibitem{GK08}
B. Guilfoyle and W. Klingenberg, {\it On area-stationary surfaces in certain neutral K\"ahler 4-manifolds}, Beitr\"age Algebra Geom. {\bf 49.2} (2008) 481--490. 

\bibitem {GK11}
{B. Guilfoyle and W. Klingenberg, {\it Proof of the Carath\'eodory Conjecture}, (2024) ArXiv Preprint: \url{https://arxiv.org/abs/0808.0851}}

\bibitem{GK19}
B. Guilfoyle and W. Klingenberg, {\it Higher codimensional mean curvature flow of compact spacelike submanifolds}, Trans. Amer. Math. Soc. {\bf 372.9} (2019) 6263--6281. 

\bibitem{GK20}
B. Guilfoyle and W. Klingenberg, {\it Fredholm-regularity of holomorphic discs in plane bundles over compact surfaces}, Ann. Fac. Sci. Toulouse Math. S\'erie 6 {\bf 29.3} (2020) 565--576.

\bibitem{GK24}
B. Guilfoyle and W. Klingenberg, {\it Proof of the Toponogov Conjecture on complete surfaces}, J. G\"okova Geom. Topol. GGT {\bf 17} (2024) 1--50.

\bibitem{GK24a} 
B. Guilfoyle and W. Klingenberg, {\it An index bound for smooth umbilic points}, (2024) ArXiv Preprint: \url{https://arxiv.org/abs/1207.5994}

\bibitem{HaS24}
A.S. Halilaj and K. Smoczyk, {\it Codimension two mean curvature flow of entire graphs}, J. Lond. Math. Soc. {\bf 110.5} (2024) e13000.

\bibitem{Ham41}
H. Hamburger, {\it Beweis einer Carath\'eodoryschen Vermutung I, II and III}, Ann. Math. {\bf 41} (1940) 63–-86, Acta. Math.{\it 73} (1941) 175–-228, Acta. Math. {\bf 73} (1941) 229-–332.

\bibitem{KMS16}
R. Kusner, A. Mondino and F. Schulze, {\it Willmore bending energy on the space of surfaces}, M.S.R.I. Emissary Spring 2016.

\bibitem{Lawson70}
H.B. Lawson, {\it Complete minimal surfaces in $S^3$}, Ann. of Math. {\bf 92.3} (1970) 335--374.

\bibitem{LiS09}
{G. Li and I.M.C. Salavessa,{\it Mean curvature flow and Bernstein-Calabi results for spacelike graphs}, VIII International Colloquium on Differential Geometry, Santiago de Compostela, Spain, 7-11 July 2008, satellite event of the 5th European Congress of Mathematics, Lopez, Jesus A. Alvarez (Ed), Garcia-Rio, Eduardo (Ed), World Scientific Publishing Company, April 2009, 164-174.}

\bibitem{LiS11}
G. Li, I.M.C. Salavessa, {\it Mean curvature flow of spacelike graphs}, Math. Z. {\bf 269} (2011) 697-–719.

\bibitem{Liu18}
K. Liu, H. Xu, F. Ye and E. Zhao, {\it The extension and convergence of mean curvature flow in higher codimension},  Trans. Amer. Math. Soc. {\bf 370.3} (2018) 2231--2262.

\bibitem{LaN24} 
S. Lynch and H. The Nguyen, {\it Convexity estimates for high codimension mean curvature flow}, Math. Ann. {\bf 388.1} (2024) 575--613.

\bibitem{MaN14}
F.C. Marques and A. Neves, {\it Min-max theory and the Willmore conjecture}, Ann. of Math. {\bf 179.2} (2014) 683--782.

\bibitem{MaN17}
F.C. Marques and A. Neves, {\it Existence of infinitely many minimal hypersurfaces in positive Ricci curvature}, Invent. Math. {\bf 209.2} (2017) 577--616.

\bibitem{Naff23}
K. Naff, {\it Singularity models of pinched solutions of mean curvature flow in higher codimension},  J. Reine Angew. Math.  {\bf 2023.794} (2023) 101--132.

\bibitem{Oh96}
Y.-G. Oh, {\it Fredholm theory of holomorphic discs under the perturbation of boundary conditions},  Math. Z. {\bf 222} (1996) 505--520.

\bibitem{Salvai05}
M. Salvai, {\it On the geometry of the space of oriented lines of Euclidean space}, Manuscripta Math. {\bf 118.2} (2005) 181--189. 

\bibitem{Salvai07}
M. Salvai, {\it On the geometry of the space of oriented lines of the hyperbolic space},  Glasg. Math. J. {\bf 49.2} (2007) 357--366.

\bibitem{Smoc11}
K. Smoczyk, {\it Mean curvature flow in higher codimension: introduction and survey}, In: B\"ar, C., Lohkamp, J., Schwarz, M. (eds) Global Differential Geometry, Springer Proceedings in Mathematics, {\bf 17} Springer, Berlin, Heidelberg (2011) 231--274.

\bibitem{Song23}
A. Song, {\it Existence of infinitely many minimal hypersurfaces in closed manifolds}, Ann. of Math. {\bf 197.3} (2023) 859--895.

\bibitem{Stahl96a}
A. Stahl, {\it Regularity estimates for solutions to the mean curvature flow with a Neumann boundary condition}, Calc. Var. Partial Differential Equations {\bf 4} (1996) 385--407.

\bibitem{Stahl96b}
A. Stahl, {\it Convergence of solutions to the mean curvature flow with a Neumann boundary condition}, Calc. Var. Partial Differential Equations {\bf 4} (1996) 421--441.

\bibitem{Top95}
V.A. Toponogov, {\it On conditions for existence of umbilical points on a convex surface}, Sib. Math. J. {\bf 36} (1995) 780--784. 

\bibitem{Will65}
T.J. Willmore, {\it Note on embedded surfaces}, An. \c{S}tiin\c{t}. Univ. Al. I. Cuza Ia\c{s}i. Mat. (N.S.) {\bf B11} (1965) 493--496.



\end{thebibliography}
\end{document}